\documentclass[a4paper]{amsart}
\usepackage{amssymb}
\usepackage{amsfonts}
\usepackage{amsmath}
\usepackage{mathtools}

\DeclarePairedDelimiter\floor{\lfloor}{\rfloor}
\usepackage{url}
\DeclareRobustCommand{\BibTeX}{%
  {\normalfont B\kern-.05em{\scshape i\kern-.025em b}\kern-.08em \TeX}%
}
\usepackage{diagrams}

\title[Finite group actions]{Finite group actions\\ on abelian groups of finite Morley rank}
\date{29 May 2023}
\author{Alexandre Borovik}
\thanks{\copyright\ 2023 Alexadre Borovik. The last submitted version of the paper without changes made by publishers.}
\address{Department of Mathematics, University of Manchester, UK}
\email{alexandre $\gg at \ll$ borovik.net}

\newtheorem{hypothesis}{Hypothesis}
\newtheorem{problem}{Problem}
\newtheorem{lemma}{Lemma}[section]
\newtheorem{theorem}{Theorem}
\newtheorem{corollary}[lemma]{Corollary}

\newtheorem{fact}[lemma]{Fact}
\newtheorem*{fact*}{Fact}

\newcommand{\bq}{\begin{quote}}
\newcommand{\eq}{\end{quote}}
\newcommand{\bi}{\begin{itemize}}
\newcommand{\ei}{\end{itemize}}
\newcommand{\bd}{\begin{description}}
\newcommand{\ed}{\end{description}}
\newcommand{\ben}{\begin{enumerate}}
\newcommand{\een}{\end{enumerate}}
\newcommand{\bbm}{\begin{bmatrix}}
\newcommand{\ebm}{\end{bmatrix}}
\newcommand{\bea}{\begin{eqnarray*}}
\newcommand{\eea}{\end{eqnarray*}}

\newcommand{\bF}{\mathbb{F}}
\newcommand{\bG}{\mathbb{G}}

\newcommand{\rk}{\mathop{\rm rk}}

\newcommand{\gl}{\mathop{\rm GL}}
\newcommand{\SL}{\mathop{\rm SL}}
\newcommand{\End}{\mathop{\rm End}}
\begin{document}

\begin{abstract}
This paper develops some general  results  about actions of finite  groups on infinite abelian groups  of exponent $p$ in the finite Morley rank category. These results are applicable to a range of problems on groups of finite Morley rank  discussed in \cite{borovik-deloro}. Also, they yield a proof of the long-standing conjecture of linearity of irreducible definable actions of simple algebraic groups on elementary abelian $p$-groups of finite Morley rank \cite[Conjecture 12]{borovik-deloro} and \cite[Question B.38]{BN-book}. Crucially, these results are  needed for the papers by Ay\c{s}e Berkman and myself \cite{berkman-borovik,berkman-borovik-solvable} where we have proved an explicit, and best possible, upper bound for the degree of generic multiple transitivity for an action of a group of finite Morley rank on an abelian group.
\end{abstract}

\maketitle

\begin{center}
{\large\emph{Dedicated to Boris Zilber who laid the path}}
\end{center}
\bigskip

\section*{Preamble}

\begin{flushright}
\emph{No man is an Iland, intire of itselfe; every man \hspace{2em}\\
is a peece of the Continent, a part of the maine.}\\
John Donne, 1623
\end{flushright}

\bigskip

 The field of study reflected in the title of  this paper could appear to be rather esoteric; however, it is a part of a much wider area of classification of \emph{simple groups of finite Morley rank}. I recommend Gregory Cherlin's informative and incisive  survey \cite{Cherlin2023} of the current state of this classification. In short, there are  three types of  connected  simple  groups of finite Morley rank: \emph{degenerate} (they do not contain involutions), \emph{odd} (contain involutions, but do not contain infinite groups of  exponent $2$), \emph{even} (contain infinite groups of exponent $2$).  Groups of even type have been identified as simple algebraic groups over algebraically closed fields of  characteristic $2$ \cite{ABC-book}. Little is known about infinite simple groups of degenerate type beyond a fantastic result by Fr\'{e}con on groups of Morley rank 3 \cite{Frecon2018} (beautifully elucidated by Corredor and Deloro \cite{Corredor-Deloro2022}).
 On the contrary, quite a lot is known about groups of odd type, but still not enough for proving for them the special case of the \emph{Cherlin-Zilber Algebraicity Conjecture}:

\begin{quote}
Simple groups of finite Morley rank and odd type are algebraic groups over algebraically closed fields of odd or zero characteristic.
\end{quote}
Since the appearance of the classification of simple groups of even type \cite{ABC-book} the study of simple groups of finite Morley rank has been moving in two streams:
\begin{itemize}
  \item[Stream 1:]  Proving the Cherlin-Zilber Algebraicity Conjecture for groups of odd type, as outlined in Cherlin's survey \cite{Cherlin2023}.
  \item[Stream 2:]  Study of actions, and specifically actions as automorphisms
of abelian groups, of groups of finite Morley rank on the basis of the knowledge
already accumulated in the efforts to classify the simple ones. This direction is motivated by the fact that groups of finite Morley rank appeared as \emph{binding groups} in model theory, hence \emph{act} somewhere. A survey of this direction can be found in \cite{borovik-deloro}.
\end{itemize}

Stream 2 was initiated by Cherlin who suggested to me to start looking for areas of possible application of the classification of groups of even type \cite{ABC-book}. After some discussion we decided to take a look at \emph{definably primitive}   permutation groups $(G, X)$, that is, definable faithful actions of $G$ on $X$ such that there are no nontrivial definable $G$-invariant equivalence relation on $X$. Since the stabilisers of points and orbits of $G$ on $X$ are definable, this means, in particular, that a definably primitive action is transitive. `Faithful' here means that only $1$ fixes all elements of $X$.
We proved the following.
\begin{quote} \textbf{Theorem} \cite[Theorem 1]{BCPermutation}
There exists a function $f: \mathbb{N} \rightarrow \mathbb{N}$ with the following property.
If a group $G$ of finite Morley rank acts on a set $X$ of finite Morley rank definably and definably primitively, then:
\[\rk(G) \leqslant f(\rk X).\]
\end{quote}
The proof of this result is an indicator of the role of the classification technique in Stream 2: an answer to a basic question about actions of groups of finite Morley rank required the use of the classification of simple groups of even type  together with  the full range of techniques developed for the ongoing study of groups of odd type.

Macpherson and Pillay \cite{MPPrimitive} (following an established tradition from finite group theory) say that  a group $G$ of finite Morley rank is of \emph{affine type} if  $G$ is a semidirect product of definable subgroups $G = V\rtimes H$, where $V$ is either elementary abelian or divisible torsion-free abelian, and the group $H$ acts on $V$ faithfully and $V$ does not leave any definable subgroup of $V$ other than $0$ or $V$ invariant. In that case the natural action of $G$ on the coset space $G/H$ is definably primitive.

This is an important class of primitive groups of permutations. In  finite group theory, this class made its first appearance in  the celebrated theorem by Galois:
\bq
A finite solvable  primitive permutation group has degree $p^k$ (that is, the set on which it acts contains $p^k$ points)  for some prime number $p$.
\eq
and was the reason why Galois constructed Galois fields \cite{Neumann2006}.

To cut a long story short, the present paper has been written because its results are essential for the study of primitive permutation groups of finite Morley rank and affine type carried out, over some years, by  Ay\c{s}e Berkman and myself \cite{berkman-borovik-sharp,berkman-borovik,berkman-borovik-solvable}. Our project calls on a surprising range of results and techniques, including almost everything which has been done so far in various approaches to the Cherlin-Zilber Conjecture. In particular, the present paper uses basic concepts from the representation theory of finite groups and associative algebras \cite{curtis-reiner,erdmann-holm} appropriately adapted for the finite Morley rank context.

Finally, a word about the epigraph from John Donne. I proved the results of the present paper in a de-facto imprisonment of a strict lockdown\footnote{I told this bizarre story  in \cite{borovik-inherited}.}. I would not even try to prove them if I did not see my work as part of a much bigger collective project. The lockdown episode reminded me that I started my work in groups of finite Morley rank 40 years ago in also almost complete isolation but got critically important help from \emph{our} (as I can now call it) community. This was how I described it in the Introduction to \cite{ABC-book}.
\begin{quote}
Vladimir Nikanorovich Remeslennikov in 1982 drew my
attention to Gregory Cherlin's paper \cite{Cherlin1979} on groups of
finite Morley rank and conjectured that some ideas from
my work [on periodic linear groups] could be used in this
then new area of algebra. A year later Simon Thomas
sent to me the manuscripts of his work on locally finite
groups of finite Morley rank. Besides many interesting
results and observations his manuscripts contained also
an exposition of Boris Zilber's fundamental results on
$\aleph_1$-categorical structures which were made known to many
Western model theorists in Wilfrid Hodges' translation of
Zilber's paper \cite{Zilber1977} but which, because of the regrettably
restricted form of publication of the Russian original,
remained unknown to me.
\end{quote}
You can learn more about this story from Wilfrid Hodges \cite{Hodges2023}; a historic perspective is presented by Bruno Poizat \cite{poizat2023}. In the early 1980s Hodges worked hard on development of links and channels of communication  between   Western and Soviet model theorists -- and, ironically, also directed me to Zilber's works. I knew Boris personally, but we lived in different cities in Siberia which was  as if we were on different planets. In years that followed I have learnt a lot from Boris,  but here I wish to emphasise perhaps the most important lesson:  the importance of looking at the wider landscapes of mathematics, something that I am trying to do in this paper.

\section{A more technical introduction}

\subsection{Ranked universes} \label{sec:Universe} First of all, we work in a ranked universe in the sense of \cite[Chapter 4]{BN-book}. In particular, a\emph{ group of finite Morley rank}  means for us  `\emph{a group definable in a ranked universe}'. When we have several algebraic structures  of finite Morley rank in the same statement, with definable actions and relations between them, it means that all these structures and relations belong to the same ranked universe $\mathcal{U}$  (which is usually not mentioned). This convention is convenient because it simplifies the language and makes arguments easily accessible to group theorists with some knowledge of the finite group theory or the theory of linear algebraic groups. So far I am aware of just one paper \cite{Tindzogho-Ntsiri2017} where the expression `a group definable in a ranked universe' is used systematically.

In applications of results of the present paper the universe $\mathcal{U}$ is usually the universe of interpretable sets of some group $\mathcal{G}$ of finite Morley rank. For example, in one of  the principal results of the present paper, Theorem  \ref{th:linearity}, the group $\mathcal{G}$ is the semidirect product $V \rtimes G$. In its turn, Theorem  \ref{th:linearity} will be applied to a point stabiliser in some generically multiply transitive permutation group $\mathcal{H}$ of finite Morley rank, see further discussion in \cite{berkman-borovik-solvable}.

\subsection{Some terminology for group actions}
We use terminology and  notation from the books \cite{ABC-book,BN-book} and keep in mind
the  ranked universe convention of Section \ref{sec:Universe}.

Let $V$   be an infinite abelian group of finite Morley rank, and  $X$ a finite set of definable isomorphisms from $V$ onto  $V$ closed under composition and inversion, so $X$ is a finite group. We shall say in this situation that the finite group $X$ acts on $V$ \emph{definably}. We include the elements of $X$  in the signature of the language and treat them as function symbols. We shall also say that $V$ is an $X$-module or that $(V,X)$ is an $X$-module.

This paper is restricted to the most important case when $V$ is elementary abelian, that is, abelian and periodic of exponent $p$ for some prime number $p$. We shall usually treat  $V$ as a vector space over the prime field $\bF_p$.

We use additive notation for the group operation in $V$. The key player in our study is the \emph{subring} $R$ generated by $X$  in the ring  $\End V$ of endomorphisms of $V$. Obviously, $R$ is finite and can be viewed as a finite dimensional $\bF_p$-algebra. It is important to observe that elements of $R$ are \emph{definable} endomorphisms of $V$. We use the usual name for $R$: it is the \emph{enveloping ring} (or \emph{enveloping algebra over $\bF_p$}) of the action of the group $X$ on $V$ and will be denoted $$R = E(X).$$

In a more general situation, if $G$ is any other group which acts on the group $V$, we say that the action of $G$ is \emph{irreducible} if $0$ and $V$ are the only $G$-invariant subgroups of $V$; we shall also say that $V$ is a \emph{simple $G$-module} or \emph{simple $R(G)$-module}. In our setup, the action of the finite group $X$ on $V$ cannot be irreducible: take $v\ne 0$, then the orbit of $v$ under the action of $X$ is finite and generates a finite $X$-invariant subgroup, while $V$ is infinite.  Therefore we need to adjust the concept of irreducibility to make it usable for the group $X$.

We shall say that the action of the group $X$ is   \emph{smooth} if any $X$-invariant \emph{connected} definable subgroup of $V$ equals $0$ or $V$, and, equivalently, that $V$ is a \emph{smooth $X$-module}. This is a very natural concept, and examples are abundant. For example, let $G = \mathop{\rm GL}_n(K)$  for  an algebraically closed field $K$ of characteristic $p>0$, acting naturally on $V=K^n$. Then the action  of $\mathop{\rm GL}_n(\mathbb{F}_{p^k}) < G$ on $V$ is smooth.
\subsection{Finite groups and Jordan properties}

Our paper starts with a discussion of the following problem, which naturally arises in \cite{berkman-borovik}.

\begin{problem} \label{prob:lower bound}
 Given a finite group $X$, find good lower bounds for the Morley ranks of faithful smooth $X$-modules of fixed positive characteristic $p$.
\end{problem}

The first result of the paper, Theorem \ref{th:lower bound} below, reduces Problem \ref{prob:lower bound} to a similar question about the minimal degree of  faithful  finite dimensional linear representations of the group $X$ over an algebraically closed field of characteristic $p$, the latter having been studied in finite group theory for quite some time \cite{Dickson1908,Landazuri-Seitz1974, Wagner1976, Wagner1977}.

Let $X$ be a finite group and $p$ a prime number. We introduce two parameters characterising its size and complexity.
\bi
\item $d_p(X)$ is the minimal degree of a faithful linear representation of $X$ over the algebraically closed field $\overline{\bF_p}$.
\item $r_p(X)$ is the minimal Morley rank of an infinite elementary abelian $p$-group $V$ of finite Morley rank such that $X$ acts on $V$ faithfully, definably, and smoothly.
\ei

\begin{theorem} \label{th:lower bound} Under the assumptions of Problem \emph{\ref{prob:lower bound}},
\[
d_p(X) = r_p(X).
\]
\end{theorem}

This theorem is one of a large body of results which establish close connections and analogies between groups of finite Morley rank, on one hand, and finite groups and algebraic groups, on another.

The following statement is an immediate corollary of Theorem~\ref{th:lower bound}  via the famous theorem of  Larsen and Pink  about  finite linear groups   \cite{Larsen-Pink2011}, see Section~\ref{sec:LP}, Fact~\ref{fact:LP}.

\begin{theorem} \label{th:-corollary-via-LP}  There is a function
\[
J: \mathbb{N} \longrightarrow \mathbb{N}
\]
with the following property.\\

If $H$ is a finite simple group which acts definably and faithfully on an infinite connected elementary abelian $p$-group of Morley rank $n$ then either
\bi
\item  $|H| \leqslant J(n)$, or
\item $H$ is a group of Lie type in characteristic $p$.
\ei
\end{theorem}

Section \ref{Jordan-type} contains  a brief discussion of Theorem \ref{th:-corollary-via-LP} and other ``theorems of Jordan type''; there is a feeling that there could be some general  model-theoretical facts underpinning them all.

\subsection{From finite groups to simple algebraic groups over algebraically closed fields} \label{sec:intermediate-preliminaries}

In Sections  \ref{sec:linearity-proof} to \ref{sec:linearity-2} we shall study definable actions of simple algebraic groups $G$ over algebraically  closed fields on elementary abelian $p$-groups of finite Morley rank.
Our approach is based on the analysis of actions on $V$ of finite subgroups of $G$, and on the use of the technique developed in Section \ref{sec:enveloping}.  Theorem \ref{fact:intermediate-subgroup} stated below is the principal tool for transfer of information on certain finite subgroups of $G$ to the group as a whole $G$ itself. The formulation of Theorem \ref{fact:intermediate-subgroup} needs to be preceded by a few words on simple algebraic groups.

First of all, there is some mismatch in the terminology: it is a traditional convention of the algebraic groups theory, that an (infinite) algebraic group $G$ is called  \emph{simple} if  $G$ is perfect, that is, $[G,G]=G$, with $G/Z(G)$ simple in the usual sense of this word  and $Z(G)$ finite. A finite group $G$ with the same properties is called in the finite group theory \emph{quasisimple}. So if $G$ is finite, `simple' really means simple: no non-trivial proper normal subgroups. Of course, every finite subgroup is algebraic,  but  it will be always clear from the context whether a particular algebraic group is finite or not.

Let $G = \mathbb{G}(K)$ be the group of $K$-points of a simple algebraic group $\bG$ defined over an algebraically closed field $K$ of characteristic $p$. In model theory, there is convention to call $G$ a simple algebraic group over $K$. By Poizat \cite{poizat1988} the group $G$ is bi-interpretable with the field $K$. On the other hand $G$ is birationally isomorphic with the group of points over $K$ of a simple algebraic group $\mathbb{H}$ defined over the prime field  $\bF_p$ \cite{Borel1970}. Since $G$ and $K$ are bi-interpretable, this isomorphism is definable in $G$ as a pure group; hence we can assume without loss of generality that $\bG$ is defined over $\bF_p$ and, for every intermediate field $\bF_p < F < K$, the group $G$ contains  the  subgroup $G(F)=\bG(F)$ of $F$-points.

Let now $K_\infty$ be  the algebraic closure of the prime field\/ $\mathbb{F}_p$ in $K$ and  $G_\infty$  the group of points of $G$ over $K_\infty$. The group $G_\infty$ is the union of finite subgroups $G(\mathbb{F}_{p^k})$ for all natural numbers $k$, hence $G_\infty$ is locally finite. As we shall see in later sections, the restriction to the group $G_\infty$ of a definable action of the group $G$  on an elementary abelian $p$-group of finite Morley rank could be studied by methods developed in Section \ref{sec:enveloping}.

\begin{theorem}  \label{fact:intermediate-subgroup}
Let $K$ be an algebraically closed field of characteristic $p > 0$ and $K_\infty$  the algebraic closure of the prime field\/ $\mathbb{F}_p$ in $K$. Let $G$ be  a semisimple algebraic group over $K$ and $G_\infty$  the group of points of $G$ over $K_\infty$. If  $M$ is  subgroup of $G$ containing $G_\infty$ and the structure $(G,M)$ has finite Morley rank, then $M=G$.
\end{theorem}

\subsection{Linearisation of actions of simple algebraic groups}

In Section \ref{sec:linearity} we prove Theorem \ref{th:linearity}. Here, an action of a group $G$ on an abelian group $V$ is called
\begin{itemize}
\item  \emph{irreducible} if $V$ contains no $G$-invariant subgroups other than $0$ and $V$, and
\item \emph{definably irreducible} if $V$ contains no $G$-invariant \emph{definable} subgroups other than $0$ and $V$.
\end{itemize}
If $G$ is a connected group of finite Morley rank, then these two properties are equivalent \cite[Lemma I.11.3]{ABC-book}.

The following Theorem \ref{th:linearity} in Section \ref{sec:linearisation} answers the long-standing conjecture of linearity of irreducible definable actions of simple algebraic groups on elementary abelian $p$-groups of finite Morley rank \cite[Question B.38]{BN-book} with further details inquired about in \cite[Conjecture 12]{borovik-deloro}. Crucially, Theorem \ref{th:linearity}(1) is  needed for the papers by Ay\c{s}e Berkman and myself \cite{berkman-borovik,berkman-borovik-solvable}
\begin{theorem}  \label{th:linearity}
Let $K$ be an algebraically closed field of characteristic $p >0$ and $G$ be  a connected algebraic group over $K$. Assume that $G$ acts definably and faithfully on an elementary abelian  $p$-group $V$ of finite Morley rank. Assume that this action is definably irreducible. THen the following is true:
\ben
\item The group $V$ has a structure of a finite dimensional $K$-vector space compatible with the action of $G$.
\item Assume in addition that $G$ is simple. Let $\widehat{G}$ be a simply connected simple algebraic group over $K$ covering $G$. Then $ \rho: \widehat{G} \longrightarrow G \hookrightarrow \mathop{GL}(V)$ is an irreducible $K$-linear representation of the  group $\widehat{G}$ on  $V$.  There are
irreducible \emph{rational} representations $\omega_1, \dots, \omega_m$ of the group $\widehat{G}$, and
there are
$(V\rtimes G)$-definable automorphisms $\varphi_1, \dots, \varphi_m$ of the field $K$,
such that $\rho = \bigotimes_{i = 1}^m \varphi_i\omega_i$. In particular, the representation $\rho$ is $(V\rtimes G)$-definable.
\een

\end{theorem}

Surprisingly, the following immediate corollary  of Theorem \ref{th:linearity}(1) for algebraic groups appears to be new. However, Adrien Deloro informed me that a special case of  this result,  where $G$ acted transitively on $V\smallsetminus \{0\}$, had been proven in 1983 by Knop \cite[Satz 1]{Knop1983}.

\begin{corollary} \label{corollary: algebraic-groups}
Let  $H= V\rtimes G$ be an algebraic group over an algebraically closed field $K$ of characteristic $p>0$, where $G$ is a connected algebraic group and $V$ is a unipotent \emph{(}written in additive notation\emph{)}.

Assume that  $V$ does not have closed $G$-invariant subgroups other than $0$ and $V$.

Then $V$ is an abelian group of exponent $p$ and has a structure of a finite dimensional vector space over $K$ invariant under the action of $G$.
\end{corollary}

Here, a \emph{unipotent} group in characteristic $p>0$ is a linear algebraic group containing only $p$-elements.

Theorem \ref{th:linearity}(2) is a more precise and detailed version of the following result by Bruno Poizat.

\begin{fact} \emph{(Poizat \cite[Theorem 2]{poizat2001})}
If $K$ is a field of finite Morley rank and non-zero characteristic $p$, any simple definable subgroup $G$ of $\mathop{\rm GL}_n(K)$ is definable in the language of the field $K$ augmented by a finite number of definable field automorphisms.
\end{fact}

In characteristic $0$, a much stronger result is known:

\begin{fact}[Combination of \cite{LWCanada,MPPrimitive, poizat2001}, and \cite{BBDefinably, DWGeometry}]
Let $(G, V)$ be a faithful, irreducible module of finite Morley rank where $G$ is infinite and $V$ is torsion-free. Then there is a definable field over which $V$ is a finite-dimensional vector space and $G$ is a subgroup of  $\gl(V)$. If in addition $G$ is simple and contains a non-identity unipotent element or an involution then $G$ is Zariski closed in $\gl(V)$.
\end{fact}

\subsection{Linear groups of finite Morley rank} Let us return to Theorem \ref{th:linearity}(1) and use notation from its statement. This theorem says that $G$ is a subgroup of the finite dimensional general linear group $\mathop{\rm GL}_K(V)$. This is a natural question:
\bq
Is the group $G$\/ Zariski closed in $\mathop{\rm GL}_K(V)$?
\eq

If the automorphisms $\varphi_1, \dots, \varphi_m$ in part (2) of Theorem \ref{th:linearity} are Frobenius maps or their inverses, then the representation $G \longrightarrow \mathop{\rm GL}_K(V)$ is rational and its image (which is $G$), is Zariski closed in $\mathop{\rm GL}_K(V)$. But here we encounter one of the oldest problems of the theory of groups of finite Morley rank.

\begin{problem}[Angus Macintyre, {\cite[Question B35, p.~364]{BN-book}}]
Can a structure of the form
\[
\langle K; +, \cdot, \varphi \rangle,
\]
where $\langle K; +,\cdot \rangle$ is an algebraically closed field of characteristc $p > 0$ and
$\varphi \in \mathop{\rm Aut}(K)$ is neither a Frobenius automorphism nor the inverse of one, have finite Morley rank?
\end{problem}

If the answer to Macintyre's problem in \emph{no}, then, in Theorem \ref{th:linearity}, the group $G$ is Zariski closed in $\mathop{\rm GL}_K(V)$. Otherwise, this is not true in general.

\subsection{Linearisation of actions of  solvable-by-finite groups}
Finally, we extend Theorem \ref{th:linearity} to solvable-by-finite groups.

\begin{theorem} \label{th:linearity-2}
Let $K$ be an algebraically closed field of characteristic $p >0$ and $G$ be  a connected algebraic group over $K$. Assume that $G$ acts definably and faithfully on an elementary abelian  $p$-group $V$ of finite Morley rank. Assume that this action is definably irreducible.
Then $G^\circ$ is a good torus and  $V$ has a definable structure of a finite dimensional $K$-vector space compatible with the action of $G$, with the field $K$ definable in $V\rtimes G$.

\end{theorem}

\section{Enveloping algebras enter the scene}
\label{sec:enveloping}

\subsection{Definitions and generalities} In this section we work under assumptions which are weaker than those of Theorem \ref{th:lower bound}:
\bi
\item $X$ is a finite group which acts, definably and smoothly, on an infinite connected elementary abelian $p$-group $V$ of finite Morley rank.
\ei

Notice that we do not assume that the action of the group $X$ is faithful. This allows us to pass these assumptions to factor modules of $V$ by definable $X$-submodules. The group $V$ is treated as a vector space over  $\bF_p$.

 \begin{lemma}
 The  canonical action of the group algebra $A=\bF_p[X]$ on $V$ is definable.
 \end{lemma}

\begin{proof}
Indeed every element from $A$ acts on $V$ as a sum of definable endomorphisms (which came from $X$) and is therefore definable.
\end{proof}

Another important player is the enveloping algebra of $X$ on $V$, that is, the ring $R$ generated in ${\rm End}_{\bF_p} V$ by elements of $X$, or, which is the same, the image of $A$ in ${\rm End}_{\bF_p} V$. We denote this ring by  $E(X)$.

Both $A$ and $R$ are finite dimensional algebras over $\bF_p$,  their action on $V$ is smooth while the action of $R$ on $V$ is also faithful.
We shall treat $V$ as a right $A$-module and right $R$-module, and, enlarging the signature of the language, we will treat elements from $A$ and $R$ as function symbols.

A finite-dimensional associative algebra over a finite field $\mathbb{F}_p$  of prime order $p$ is the same as a finite ring of characteristic $p$. Their structure is of course well known. We will need a definition of the \emph{Jacobson radical} $J(R)$ of a finite dimensional algebra $R$ over a field: $J$ is
the intersection of all maximal left ideals of $R$. It can be proved that  $J$ is an ideal, and, moreover, $J$ can be characterised as the set of all elements $r\in R$ such that $Mr = 0$ for every simple (or irreducible, which is the same) $R$-module $M$.

\begin{fact}[{Wedderburn-Maltsev Theorem, \cite[Theorem VI.2.1]{Drozd-Kirichenko1994}}] \label{fact:finite-algebras} Let $R$ be a   finite-dimensional associative algebra with identity $1$ over a finite field\/ $\mathbb{F}_p$  of prime order $p$ and $J$ its Jacobson radical. Then
\bi
\item[(a)]
$R = J+S$ where $S$ is a semisimple algebra, $J\cap S = 0$ and $S$ is the direct sum of matrix algebras
\[
S = S_1 \oplus\cdots\oplus S_k, \;\; S_i \simeq M_{d_i\times d_i}(\mathbb{F}_{p^{m_i}}), \;\; i = 1,2, \dots, k.
\]

\item[(b)]Let $Q = 1+J$. Then  $Q$  is a normal $p$-subgroup in the  group of units $R^*$ of $R$. Moreover,  $R^*$ is a semidirect product  $R^* = Q \rtimes S^*$ of $Q$ and the  group of units $S^*$ of $S$.  In particular,
    \[
S^* \simeq {\gl}_{d_1}(\mathbb{F}_{p^{m_1}}) \times\cdots\times {\gl}_{d_k}(\mathbb{F}_{p^{m_k}}).
    \]
\ei
\end{fact}

\subsection{The case of smooth action} The arguments below freely use, without specific references, definitions and results from more basic parts of the theory of modules which can be found, for example, in \cite{erdmann-holm}.

 The group $X$ will not be mentioned in the rest of this section. Rather, we will work under the following hypothesis.

\begin{hypothesis} \label{hypo}
In the notation of Fact~{\rm \ref{fact:finite-algebras}}, $R$ is a finite dimensional\/ $\bF_p$-algebra acting definably, smoothly, and faithfully on an  infinite elementary abelian $p$-group $V$ of finite Morley rank.
\end{hypothesis}

\begin{lemma} \label{lemma:semisimple}
Under  Hypothesis\emph{ \ref{hypo}}, the algebra $R$ is semi-simple and\/ $V$ is a semisimple $R$-module.
\end{lemma}

\begin{proof}

We work in the notation of Fact \ref{fact:finite-algebras}. Since $VQ$ is a definable abelian-by-finite $p$-group, $VQ$ is nilpotent. By properties of commutators in  groups of finite Morley rank, $VJ =  [V, Q] < V$ is a proper connected definable subgroup  of $V$ invariant under $R$, hence $V J=0$. But then  $J=0$  and $R$ is semisimple, therefore $V$ is also semisimple, that is,  a direct sum of simple $R$-modules
\begin{equation} \label{eq:direct_decoposition}
V = \bigoplus_{\ell\in L} U_\ell \; \mbox{ for some index set } L.
\end{equation}
\end{proof}

\begin{theorem} \label{theorem:action_of_E} Under  Hypothesis \emph{\ref{hypo}},

\bi
\item[(a)] All simple factors in the direct sum of Equation  {\rm (\ref{eq:direct_decoposition})} are isomorphic.
\item[(b)] The algebra  $R$ is simple and therefore is isomorphic to an algebra of all matrices of size $\ell \times \ell$, for some $\ell$, over a finite extension of\/ $\bF_p$.
\ei
\end{theorem}

\begin{proof}
(a) Denote by $\mathcal{I}= \mathcal{I}(R)$ the set of isomorphism classes of simple $R$-modules; notice that the set $\mathcal{I}$ is finite.    We collect isomorphic simple factors and rewrite the direct sum  of simple $R$-modules in  Equation  (\ref{eq:direct_decoposition}) as
\begin{equation} \label{factors_isomorphic}
V = \bigoplus_{I\in \mathcal{I}} U_I.
\end{equation}
where all simple summands of $U_I$ belong to the isomorphism class $I$.

We want to prove that each submodule  $U_I$ is definable; then it will  follow from the smoothness of the action of $R$ on $V$ than $V = U_I$ for some $I$, that is, that all summands in (\ref{eq:direct_decoposition}) are isomorphic.

For that purpose,  an element $v\in V$  will be  called \emph{$I$-cyclic} for $I \in \mathcal{I}$ if all simple summands in the cyclic module $vR$ belong to $I$. If $u,v$ are $I$-cyclic elements of $V$, then $(u+v)R \leqslant uR + vR$, and all simple summands in $uR + vR$, and hence in $(u+v)R$, belong to $I$ and therefore $u+v$ is an $I$-cyclic element. It follows that the set of all $I$-cyclic elements in $V$ coincides with  $U_I$.

 Let us denote by $\mathcal{K}_I$ the set of all right ideals $K$ in $R$ such that all simple summands of the factor module $R/K$ belong to $I$. Then $U_I$ is defined by the formula
\[
\Phi(v) := \bigvee_{K\in  \mathcal{K}_I} \left(\left( \bigwedge_{k\in K} vk=0 \right) \wedge \left(\bigwedge_{l\in R\smallsetminus K} vl \ne 0 \right)  \right).
\]
This completes the proof of part (a).

(b) Now $R$ is also the enveloping algebra of its restriction to every simple summand of $V$ and is therefore simple.
\end{proof}

\subsection{The weight decomposition for a coprime action of a finite abelian group and the Multiplicity Formula} \label{sec:weight-decomposition}

Now we will focus our attention  temporarily on actions of finite abelian groups of orders coprime to $p$, and reformulate the previous results in more familiar terms, in this special case.

Let  $V$ be a connected $H$-module of characteristic $p>0$ with  $H$ a finite abelian group  of order coprime to $p$.

View $V$ with the action of $H$ as a module over the finite group algebra $A = \bF_p[H]$. Applying Maschke's Theorem to the action of $H$ on $A$ by multiplication, we see that $A$ is semisimple and is a direct sum of simple finite commutative algebras, that is, finite  fields (of course, of characteristic $p$).

We shall call a non-zero element $v\in V$ a \emph{weight element} if ${\rm Ann}_A(v)$ is a maximal ideal in $A$; equivalently, this means that $vA$ is  an irreducible $A$-module. It follows that if $F = A/{\rm Ann}_A(v)$ and $\lambda: A \longrightarrow F$ is the canonical homomorphism, then $vA$ is a $1$-dimensional vector space over $F$ (notice that $F$ may be bigger than $\bF_p$) and, for every $a \in A$,  we have $$va= \lambda(a) v,$$ where the left-hand side is  understood in the sense of a (right) $A$-module, and the right-hand side is a vector space over the field $F$. This justifies the homomorphisms $\lambda$ being called \emph{weights} of $A $.

Observe that when restricted to $H$, weights become characters of $H$, that is, homomorphisms from $H$ to the multiplicative group $\overline{\bF}_p^*$ of the algebraic closure $\overline{\bF}_p$ of the prime field $\bF_p$. Also it is easy to see that $F \simeq \bF_p[\lambda[H]]$, where $$\lambda[H] = \{\, \lambda(h): \; h \in H\,\}.$$

Obviously, if $u,v \in V$ are weight vectors for the same weight $\lambda$  then either $u+v=0$ or $u+v$ is a weight vector for $\lambda$. Hence all such vectors form a \emph{definable} $A$-submodule $V_\lambda \leqslant V$, and it follows from Maschke's Theorem that
\[
V = \bigoplus V_\lambda,
\]
where the direct sum is taken over all weights of $H$ on $V$. It follows that all weight spaces $V_\lambda$ are connected, and their total number does not exceed $n = \rk V$.
It is also useful to keep in mind that $V_\lambda$ is a vector space under the action of the finite field $F_\lambda = \bF_p[\lambda[H]]$. Observe further that $F_\lambda$ is the enveloping algebra for the action of the group $H$ on $V_\lambda$.

If we call $\rk V_\lambda$ the \emph{multiplicity} of the weight  $\lambda$ then, as one would expect, we have the following.

\begin{theorem}[Multiplicity Formula] \label{theorem:multiplicity}
The sum of multiplicities of weights of $H$ on $V$ equals $\rk V$.
\end{theorem}

This statement is called a theorem only because of its importance; its proof is obvious.

\subsection{Proof of Theorem \ref{th:lower bound}}

\begin{proof} To prove Theorem \ref{th:lower bound}, it would suffice to show that, for any definable and faithful smooth $X$-module $V$, $ \rk V \geqslant \dim_{\overline{\bF}_p} W$ for some faithful $\overline{\bF}_p[X]$-module $W$ on which $X$ acts faithfully.

So let $V$ be a definable and faithful smooth  $X$-module of smallest possible Morley rank and $R$ the enveloping algebra of this action.  By Theorem \ref{theorem:action_of_E},  $V$ is a direct sum of isomorphic  simple (in particular, finite) $R$-modules; obviously, $X$ acts on each of them faithfully.

Now we can look at one of these simple $R$-modules, say $U$. Assume that $\dim_{\bF_p}U = n$. By Theorem \ref{theorem:action_of_E}, $R$  is the matrix algebra ${\rm Mat}_{m\times m}(\bF_{p^l})$, where $n=m\times l$. In particular, we have a definable action of $R^*={\rm GL}_m(\bF_{p^l})$ on $U$, and, since $V$ is a direct sum of isomorphic copies of $U$, $R^*$ acts definably on $V$.  The maximal torus $H$ of $R^*$ has $m$ different weights on $U$ and therefore on $V$. From the Multiplicity Formula (Theorem \ref{theorem:multiplicity}) applied to the action of  $H$ we have that
\[
\rk V \geqslant m.
\]
But $R^*$, and hence the (homomorphic) image of  $X$ in $R^*$ has a faithful linear representation over $\overline{\bF}_p$ of degree $m$. Hence $\rk V \geqslant d_p(X)$, which completes the proof of Theorem \ref{th:lower bound}.
\end{proof}

\subsection{Theorem ~\ref{th:lower bound}: comments}

It looks as though the proof of  Theorem \ref{th:lower bound} does not use all axioms of finite Morley rank (as given in \cite[Section I.2.1]{ABC-book} and \cite[Section 4.1.2]{BN-book}); it would be interesting to find weaker conditions on the module $V$ under which Theorem \ref{th:lower bound} still holds, in a way similar to that of \cite{corredor-deloro-wiscons}.

I expect that the method outlined here gives also new approaches to some of the problems listed in the survey paper by Adrien Deloro and myself \cite{borovik-deloro}.

\section{Proof of Theorem \ref{th:-corollary-via-LP} and Jordan properties}
\label{sec:LP}

\subsection{A Larsen and Pink type theorem}

\begin{fact} \emph{(Larsen and Pink \cite[Theorem 0.2]{Larsen-Pink2011}) }\label{fact:LP}
For every $n$ there exists a constant $J '(n)$ depending only on $n$ such that any finite subgroup $\Gamma$  of ${\rm GL}_n$ over any field $k$ possesses normal subgroups
$\Gamma_3 \leqslant \Gamma_2 \leqslant \Gamma_1$ such that
\bi
\item[(a)] $|\Gamma : \Gamma_1|   < J'(n)$.
\item[(b)] Either $\Gamma_1= \Gamma_2$,  or $p := {\rm char}(k)$ is positive and  $\Gamma_1 / \Gamma_2$ is a direct product of
finite simple groups of Lie type in characteristic $p$.
\item[(c)] $\Gamma_2 / \Gamma_3$ is abelian of order not divisible by ${\rm char}(k)$.
\item[(d)] Either  $\Gamma_3 = \{1\}$  or $\ p := {\rm char}(k)$ is positive and $\Gamma_3$ is a $p$-group.
\ei
\end{fact}

Theorem \ref{th:lower bound} and Fact \ref{fact:LP} immediately yield the following result.

\begin{theorem} \label{th:version-of-LP} There is a function
\[
J: \mathbb{N} \longrightarrow \mathbb{N}
\]
with the following property:

If $X$ is a finite group which acts definably and faithfully on an infinite connected elementary abelian $p$-group of Morley rank $n$ then $H$
possesses normal subgroups
$X_3 \leqslant X_2 \leqslant X_1$ such that
\bi
\item[(a)] $|X : X_1|   <  J(n)$.
\item[(b)] Either $X_1= X_2$,  or   $X_1 / X_2$ is a direct product of finite simple groups of Lie type in characteristic $p$.
\item[(c)] $X_2 / X_3$ is abelian of order not divisible by $p$.
\item[(d)] $X_3$   is a $p$-group.
\ei
\end{theorem}

Now Theorem \ref{th:-corollary-via-LP} is a special case of Theorem \ref{th:version-of-LP}. \hfill $\square$

\subsection{Theorems of Jordan type} \label{Jordan-type}

Fact \ref{fact:LP}  and Theorem \ref{th:version-of-LP}  can be called \emph{theorems of Jordan type} since they follow the paradigm set by Camille Jordan in his famous theorem of 1878:

\begin{fact} \cite[p. 114]{Jordan1878}

There is a function
\[
J: \mathbb{N} \longrightarrow \mathbb{N}
\]
with the following property: every  finite subgroup of ${\rm GL}_n$ over a field of characteristic zero possesses an abelian normal subgroup
of index $\leqslant J(n)$.
\end{fact}

Breuillard \cite{Breuillard2022} gave an exposition  of Jordan's original proof in the modern terminology; Collins \cite{Collins2007} found the optimal explicit bound for $J(n)$.

The \emph{Introduction} to Guld \cite{Guld2020} contains an impressive survey of theorems of Jordan type for finite subgroups of groups arising in  complex algebraic geometry and in differential geometry. Due to this assumption, \cite{Guld2020} contains a more specific and narrower definition of a \emph{Jordan group}:
\bq
A group $G$ is called \emph{Jordan}, \emph{solvably Jordan} or \emph{nilpotently Jordan
of class at most} $c$  ($c \in\mathbb{N}$) if there exists a constant $J \in \mathbb{N}$ such that every finite
subgroup $X \leqslant  G$ has a subgroup $Y \leqslant X$ such that  $|X : Y| \leqslant  J$ and $Y$ is abelian,
solvable or nilpotent of class at most $c$, respectively.
\eq

\begin{fact} \cite[Theorem 2]{Guld2020} \label{fact:Guld} The birational automorphism group of a variety over a field of characteristic zero is nilpotently Jordan of class at most two.
\end{fact}

There are several variations of definitions of Jordan groups, so it could be more useful to speak about all of them as \emph{Jordan properties}.
Bandman and Zarhin \cite{Bandman-Zarhin2023} is a survey of results on Jordan properties in automorphism groups of some structures of K\"{a}hler geometry. Some results on Jordan properties in positive characteristics and further references can be found  in  \cite{Chen-Shramov2022,Prokhorov-Shramov2022}.

These results, together with many other results quoted in \cite{Guld2020}, create a feeling that there could be some underlying model-theoretic concepts and results underpinning them all.

\section{Linearisation of the actions of algebraic groups} \label{sec:linearity}

In this section, we will prove  Theorems \ref{fact:intermediate-subgroup},    \ref{th:linearity}, and  \ref{th:linearity-2}.

\subsection{Proof of Theorem \ref{fact:intermediate-subgroup}.}  \label{sec:intermediate-subgroup}

We use definitions and terminology of Section \ref{sec:intermediate-preliminaries}.

\bigskip
\noindent
\textbf{Theorem \ref{fact:intermediate-subgroup}.}
\emph{Let $K$ be an algebraically closed field of characteristic $p > 0$ and $K_\infty$  the algebraic closure of the prime field\/ $\mathbb{F}_p$ in $K$. Let $G$ be  a semisimple algebraic group over $K$ and $G_\infty$  the group of points of $G$ over $K_\infty$. If  $M$ is  subgroup of $G$ containing $G_\infty$ and the structure $(G,M)$ has finite Morley rank, then $M=G$.}

\begin{proof}
 It obviously suffices to consider only the case when $G$ is simple.  Let $T_\infty$ be a maximal torus in $G_\infty$. Its Zariski closure in $G$ is a maximal torus in $G$, let us denote it $T$. Obviously,
\[
T_\infty  \leqslant M \cap T \leqslant T,
\]
with $M \cap T$ being a definable subgroup. By Poizat \cite{poizat1988} the simple algebraic group $G$ and the field $K$ are be-interptetable. This allows us to apply \cite[Proposition I.4.20 and Lemma I.4.21]{ABC-book}, and prove that $T$ is a \emph{good torus} in the sense of  \cite[Section I.4.4] {ABC-book}, that is, every definable subgroup of $T$ is the definable hull of its torsion part. In particular,  $T$ is the definable hull of its torsion $T_\infty$.   Hence $M\cap T = T$ which means that $T \leqslant M$.
Define
\[
N = \langle T^m \mid m\in M \rangle.
\]
Being generated by Zariski closed connected subgroups, $N$ is Zariski closed. Obviously, $G_\infty = \langle T_\infty^g \mid g\in G_\infty  \rangle \leqslant N$. But the Zariski closure of $G_\infty$ in $G$ is $G$, hence $N=G$ and therefore $M=G$.
\end{proof}

\medskip
\noindent
\textbf{Remark.} The bi-interpretability of $G$ and $K$ is the cornerstone of the proof. Indeed, if $K$ is an algebraically closed field and $T = K^*$ is its multiplicative group, then the torus $T$, \emph{viewed as a pure group}, in absence of the field $K$, is not a good torus: it is easy to see that the structure $(T_\infty,T)$ has finite Morley rank.

\subsection{Proof of Theorem  \ref{th:linearity}}  \label{sec:linearity-proof}  This is the core of the paper.

\medskip
\noindent
\textbf{Theorem \ref{th:linearity}}
\emph{Let $K$ be an algebraically closed field of characteristic $p >0$ and $G$ be  a connected algebraic group over $K$. Assume that $G$ acts definably on an infinite connected elementary abelian  $p$-group $V$ of finite Morley rank. Assume that this action is definably irreducible. Then the following statements are true.
\ben
\item The group $V$ has a structure of a finite dimensional $K$-vector space compatible with the action of $G$, so the group $G$ could be viewed as a subgroup of $\mathop{GL}(V)$.
\item  Let $\widehat{G}$ be a simply connected simple algebraic group over $K$ covering $G$. Then $ \rho: \widehat{G} \longrightarrow G \hookrightarrow \mathop{GL}(V)$ is an irreducible $K$-linear representation of the  group $\widehat{G}$ on  $V$.  There are
irreducible \emph{rational} representations $\omega_1, \dots, \omega_m$ of the group $\widehat{G}$, and
there are
$(V\rtimes G)$-definable automorphisms $\varphi_1, \dots, \varphi_d$ of the field $K$,
such that $\rho = \bigotimes_{i = 1}^d \varphi_i\omega_i$. In particular, the representation $\rho$ is $(V\rtimes G)$-definable.
\een}

The proof of this theorem will spread over Sections \ref{sec:linearisation}
--\ref{Proof-Theorem-4(2)}.

It is useful to remember the ranked universe convention of Section \ref{sec:Universe}.

\medskip

\subsection{Linearisation Theorem} \label{sec:linearisation}

Recall that if $G$ is a group of finite Morley rank acting definably  on an abelian group $V$  of finite Morley rank, then the action is called \emph{definably irreducible} if the only $G$-invariant definable subgroups in $V$ are $0$ and $V$.

\begin{fact} [Linearisation Theorem]
\label{Fact:DW}
 Let $V$ be an infinite elementary abelian $p$-group of finite Morley rank and $G$ an infinite group of finite Morley rank acting on $V$ faithfully, definably, and definably irreducibly. Let $D$ be the ring of all definable endomorphisms of\/ $V$ and  $Z=C_D(G)$.  Assume that $Z$ is infinite. Then

\ben

\item  $Z$ is an algebraically closed  field definable in $V\rtimes G$  and the action of $Z$ on $V$   gives $V$ a structure of a finite dimensional $Z$-vector space (with a $Z$-linear action of $G$).

\item The enveloping algebra \emph{(}over $Z$\emph{)} $R =R(G)$ is the full matrix algebra ${\rm End}_Z(V)$.
\item $R$ is definable in $V\rtimes G$.
\een
\end{fact}

\begin{proof}
Clause (1) is a result by  Macpherson and Pillay \cite[Theorem 1.2]{MPPrimitive}, with a more complete proof given by Deloro  in  \cite{Deloro-Bogota}. It also follows from a more general and very illuminating treatment of linearisation of actions of finite Morley rank given in Deloro's ``Zilber's  skew field lemma'' \cite{Deloro2022}.
Clause (2) follows from basic algebra: an irreducible subgroup $G \leqslant {\rm GL}_n(Z)$  contains $n^2$ matrices linearly independent over $Z$ and forming a basis of the matrix algebra $M_{n\times n}(Z)$. Clause (3) follows from (1).
\end{proof}

\subsection{Groups of units of associative algebras over finite fields: ranks}

Recall that, for a prime number $r$, the $r$-rank $m_r(G)$ of a finite group $G$ is the minimal number of generators in a maximal elementary abelian $r$-subgroup of $G$. When $p>2$ we are interested in the case $r=2$ because elements of order $2$ in ${\rm GL}_n(\mathbb{F}_p)$ have eigenvalues $\pm 1$ which, of course, belong to $\mathbb{F}_p$ and therefore have a nice and easy to control behaviour.  When $p=2$, we use $r=3$, since this still gives us some degree of control. The two clauses (a) and (b) in the following Fact \ref{fact:finite-algebras-ranks} correspond to the cases when elements of order $2$ are semisimple ($p >2$) or unipotent (p=2).

\begin{fact} \label{fact:finite-algebras-ranks}  Let $R$ be a   finite-dimensional associative algebra over a finite field $\mathbb{F}_p$  of prime order $p$ and $J$ its radical. In the notation of Fact~\emph{\ref{fact:finite-algebras}}, the following hold.

\bi
\item[(a)] Assume that $p >2$.  Then
\[
m_2(R^*) = d_1+ d_2  +\cdots + d_k.
\]

\item[(b)] If $p=2$ then
\[
m_3(R^*) \leqslant \floor* {\frac{ d_1}{2}} +  \floor* {\frac{ d_2}{2}}  +\cdots + \floor* {\frac{ d_k}{2}}.
\]

\ei

\end{fact}

\begin{proof}  It is easy: for proving (a),  perhaps a simple reference to the proof of Theorem \ref{theorem:multiplicity} above would suffice. For (b), an proof follows from an elementary fact from the finite  group theory:  a cyclic group of order $3$ has only one faithful irreducible representation over the field $\mathbb{F}_2$, and it is of dimension $2$.
\end{proof}

\subsection{Proof of Theorem \ref{th:linearity} part (1)}

We start with a few general observation. If the connected algebraic group $G$ has a non-trivial unipotent radical $U\ne 1$, then $C_V(U) \ne 0$ is a proper definable $G$-invariant subgroup of $V$, which contradicts the assumptions of the theorem. Hence $U=1$ and $G$ is reductive. Set $Z= Z(G)$. If $Z$ is infinite then the theorem follows from Fact \ref{Fact:DW}. So we can assume without loss of generality that $G$ is semisimple.

Let $K_\infty\leqslant K $ be the algebraic closure of the prime field $\bF_p$, $G_\infty = \mathbb{G}(K_\infty)$, and $R_\infty = E(G_\infty)$ the enveloping algebra of $G_\infty$.

\begin{lemma} \label{blast-from-the-past} $R_\infty$ is the matrix algebra ${\rm M}_{d\times d} (K_\infty)$ for some natural number $d$.
\end{lemma}

\bigskip

\subsection{Proof of Lemma~\ref{blast-from-the-past}}

For a subgroup $H \leqslant G_\infty$, we denote by $^{u}\!H$ the subgroup generated in $H$ by all unipotent elements in $H$.

 We analyse the series of subgroups in $G_\infty$:
\[
X_k ={^{u}}\!G(\mathbb{F}_{p^{(m+k)!}}), \;  k = 1,2, \dots
\]
Obviously, they form a chain
\[
X_1 < X_2  < X_3 < \dots .
\]
and
\[
\bigcup_{i=1}^\infty X_i= G_\infty
\]
(since $G_\infty$  is generated by unipotent elements).

It is well-known that the groups $X_k$, $k = 2,3, 4, \dots$ (that is, with the exception of a few very small groups), are perfect, $X_k = X'_k$, and therefore

\bi
\item[(a)] The groups $X_k$ for $k > 1$ have no nontrivial characters $X_k  \longrightarrow \overline{\bF}_p^*$.
\ei

Let $R_k = E(X_k)$ be the enveloping algebra of $X_k$ in its action on $V$, then
\[
R_1 \leqslant R_2 \leqslant R_3 \leqslant \dots\]
and
\[
R_\infty = \bigcup_{i=1}^\infty R_i= E(G_\infty).
\]

Denote by $J_i$ the Jacobson radical of $R_i$,  $i =1, 2,\dots$. Then $R_i/J_i$ is semi-simple and
\[
R_i/J_i = M_{i1} \oplus \cdots \oplus M_{ik_i}
\]
where $M_{ij}$ are matrix algebras over finite fields $F_{ij}$ of degree $d_{ij}$. Notice that, in view of claim (a) above,  $d_{ij} > 1$  for $i>1$,  since $X_i$ have no non-trivial actions in dimension $1$.

For $p >2$  notice further that   $d_{ij}$ coincides with the $2$-rank of the corresponding group ${\rm GL}_{d_{ij}}(F_{ij})$ of invertible elements (Fact~\ref{fact:finite-algebras-ranks} (a)). Therefore

\bi
\item[(b)] For $p>2$, each group $R^*_i$ contains an elementary abelian $2$-subgroup of $2$-rank
    \[ d_i = d_{i1}+d_{i2}+\cdots +d_{ik_i}.\]
\ei

Again applying Fact~\ref{fact:finite-algebras-ranks} (a) we see that for $p>2$  the $2$-ranks $d_i$ are bounded by $\rk V$ in view of Theorem \ref{theorem:multiplicity}; for $p=2$  we similarly have, from Fact~\ref{fact:finite-algebras-ranks} (b),  that $d_i \leqslant 3\rk V$.   Hence

\bi
\item[(c)] The numerical parameters: $d_i$, $k_i$, $d_{i1},d_{i2},\dots, d_{ik_i}$ stabilise  starting from some $i_*$ as $i$ grows  and remain the same for all $i  \geqslant  i_*$.
    \ei

From that index  $i_*$ on, embeddings $R_i \leqslant R_j$  for $i < j$ can be much better controlled. Indeed,

\bi
\item[(d)] After appropriately changing the numeration we have embeddings of rings $M_{il} \leqslant M_{jl}$ for $i_* \leqslant i < j$ where the dimensions $d_{il}$ and $d_{jl}$ of $M_{il}$ and    $M_{jl}$, correspondingly, over their centers, correspondingly, are equal: $d_{il} = d_{jl}$.
\ei

This leads to
\bi
\item[(e)]  $J_i \leqslant J_j$ for $i_* \leqslant i < j$.
\ei
Indeed, if $J_i \not\hspace{-0.1ex}\leqslant J_j$ then $Q_i \not\hspace{-0.1ex}\leqslant Q_j$ and in one of the factor rings $M_{jl}$ the image $Q  = \overline{Q}_i \ne 0$. But $\overline{Q}$ is normalised by invertible elements of $M_{il}$. Let us focus on groups of units and  denote $d_{il}=d_{jl} = m$. Since the groups $X_i$ and $X_j$ are perfect,  we see a group $H_i =\SL_m(K_i)$ (which contains $X_i$)  imbedded into $H_j=\SL_m(K_j)$ (which contains $X_j$), where $K_i$ and $K_j$ are two finite fields of characteristic $p>0$, and  $H_i$ is normalising a non-trivial $p$-subgroup $Q$ in $H_j$, that is, $H_i \leqslant N_{H_j}(Q)$ and is contained in a proper parabolic  subgroup of $H_j$ which is obviously impossible: $H_i$ and $H_j$ have the same Lie rank $m-1$, but proper parabolic subgroups in $H_j$ have smaller Lie rank.

Now we have to take a look at the action of $J_i$ on $V$.
Obviously  $VJ_i = [V,Q_i]$. The group $V\rtimes Q_i$ is nilpotent and $V$ is its connected component and therefore, by standard properties of nilpotent groups of finite Morley rank, $[V,Q_i]$ is a connected definable proper subgroup of $V$. Hence

\bi
\item[(f)]  $VJ_i$ is a connected definable proper subgroup of $V$.
\ei
Now we immediately have
\bi
\item[(g)] $J_\infty =  \bigcup_{i=i_*}^\infty J_i$
is  a nilpotent ideal of $R_\infty$. Moreover, $VJ_\infty$ is a connected definable proper subgroup of $V$.
\ei

If $W = VJ_\infty \ne 0$ then $W$ is a definable proper connected $R_\infty$-invariant, and hence $G_\infty$-invariant, subgroup of $V$.
 But then $N_G(W)$ is a definable subgroup of $G$, and, of course, contains $G_\infty$; hence, by Theorem \ref{fact:intermediate-subgroup}, $N_G(W) = G$, which contradicts  irreducibility of $G$ on $V$. This proves

\bi
\item[(h)]  $J_\infty = 0$
\ei

We can now complete the proof of the lemma.  By Steps (h) and (d), $R_\infty$ is semisimple and
    \[
    R_\infty  = M_1\oplus \cdots \oplus M_k
    \]
is the direct sum of matrix algebras of degrees $d_j$  over the field $K_\infty$.

Assume that $k >1$. then
\[
V = VM_1 \oplus \cdots\oplus VM_k,
\]
where each submodule  $VM_j$ is annihilated by $M_l$ for $l \ne j$, and, moreover,
\[VM_J = \bigcap_{l\ne j}\mathop{\rm Ann}_V(M_l).
\]
By the chain condition,   $VM_j$ is the annihilator of a finite set of matrices, therefore it is  definable. Moreover,  $VM_j$ is normalised by $G_\infty$, therefore  $G_\infty \leqslant N_G(VM_j)$ and the latter is a definable subgroup, which again leads to a contradiction with Theorem \ref{fact:intermediate-subgroup} and irreducibility of $G$ on $V$. Hence $k=1$ and $R_\infty = M_{d\times d}(K_\infty)$. \hfill $\square$

\bigskip

\subsection{Back to proof of Theorem \ref{th:linearity} part (1)} \label{sec:action-of-Z}

Now we can use Lemma \ref{blast-from-the-past}. Let $Z_\infty \simeq K^*_\infty$ be the center of $R_\infty$.  Take $z \in Z_\infty$, $z \ne 0, 1$.  Then
    \[
z = g_1 + \cdots + g_n
    \]
 for some  $g_i \in G_\infty$. Some elements  $g\in G$ `commute' with $z$ in the sense that
 \begin{equation} \label{eq:commutativity-is-defianble}
 vzg = vgz   \quad\mbox{ for all} \quad v \in V.
 \end{equation}
 We denote by $M$ the set of such elements in $G$; it is easy to see that this is a subgroup.
And here is the  key observation: Equation \ref{eq:commutativity-is-defianble} can be written as a first order statement in the group language in the group $V \rtimes G$:
 \[
vg_1g+\cdots + vg_ng =  vgg_1 +\cdots + vgg_n   \quad\mbox{ for all} \quad v \in V.
\]
 Now the subgroup $M$  is definable in $V \rtimes G$, and contains $G_\infty$, so we have the following chain of subgroups:
 \begin{equation} \label{eq:intermediate1}
 G_\infty \leqslant M \leqslant G.
 \end{equation}
 The group $G$, being a semisimple linear algebraic over $K$, is decomposed as a central product
 \[
 G = G^1*\cdots * G^\ell
 \]
 of simple algebraic groups $G^i$ over $K$, $i = 1,\dots, \ell$. If $G^i_\infty$ is the group of points of $G^i$ over the field $K_\infty$, then
 \[
 G^i_\infty = G^1_\infty * \cdots * G^\ell_\infty
 \]
 and $M_i = M \cap G^i$ and in every $G^i$ we have a chain of subgroups
 \[
  G^i_\infty \leqslant M_i \leqslant G^i.
 \]

Now, Theorem \ref{fact:intermediate-subgroup} gives us $M_i = G^i$ for all $i$, hence $M=G$. Recall that the subgroup $M$ was constructed from some element $z \in Z$.
This arguments applies to all $z \in Z_\infty$, hence $Z_\infty$ and $G$ commute elementwise as multiplicative subgroups of $\mathop{\rm End} V$. The application of Fact \ref{Fact:DW} completes the proof of Theorem \ref{th:linearity}(1). \hfill $\square$

\subsection{Theorem \ref{th:linearity} part (2): some preparatory comments} \label{Proof-Theorem-4(2)-preparations}

We continue to work in notation of  Theorem \ref{th:linearity} but need additional definitions and facts about simple algebraic groups over algebraically closed fields.

If $G$ and $H$ are two simple algebraic groups over an algebraically closed field $K$ of characteristic $p >0$ which is fixed in this Section, a surjective rational homomorphism  $\zeta: G \longrightarrow H$ is called an \emph{isogeny}. It is known that $ker\, \zeta \leqslant Z(G)$, is finite and has order  coprime to $p$. Among simple algebraic groups of the same type as $G$, that is,  with the same root system, there is the group $\widehat{G}$, called \emph{simply connected}, which has a isogeny onto any other simple group of the same type, and the \emph{adjoint} group $\bar{G}$ such that any group of the same type has an isogeny onto $\bar{G}$.

Let $\Phi$ be the root system of $G$; we can select root subgroups
\[
X_r = \{x_r(t),\; t \in K\}, \; r \in \Phi,
\]
so that all of them are  defined over the prime field. An isogeny $\zeta: G \longrightarrow H$ maps the system of root subgroups $\{ X_r, \; r\in \Phi\}$, to a similar system in $H$.

If  now $\varphi \in \mathop{\rm Aut} K$ is an automorphisms of the field $K$ then it induces a \emph{field} \emph{automorphism}  of $G$ by mapping elements of root subgroups of $G$,
\[
x_r(t) \mapsto x_r(t^\varphi), \quad r\in \Phi, \quad t \in K.
\]
This gives us an automorphism $\widetilde{\varphi}$ of $G$, and similarly for $H$. Obviously, the $\widetilde{\varphi}$ commute with the isogeny, $\widetilde{\varphi}\zeta = \zeta\widetilde{\varphi}$ which justifies the use of the notation $\widetilde{\varphi}$ on the both groups $G$ and $H$. Also, the action of $\widetilde{\varphi}$ on the elements of the root subgroup $X_r = \{x_r(t),\; t \in K\}$ is the same as the action of $\varphi$ on the elements of $K$. Therefore if $\widetilde{\varphi}$ is a definable automorphism of $G$, then $\varphi$ is a definable automorphism of $K$.

The automorphism $\widetilde{\varphi}$ can be used to convert each $G$-module $W$ into another $G$-module, denoted $W^\varphi$, by the rule $wg := w(g^{\widetilde{\varphi}})$, $w\in W$, $g \in G$ \cite[\S 5]{Steinberg1963}.

\subsection{Proof of Theorem \ref{th:linearity} part (2)} \label{Proof-Theorem-4(2)}

\begin{proof} We know from part (1) of Theorem \ref{th:linearity}  that $V$ is a finite dimensional vector space over the field $K$ and $G$ is a definable subgroup in  $\mathop{\rm GL}_K(V)$.

Let $\widehat{G}$ be a simply connected simple algebraic group over $K$ and $\rho: \widehat{G} \longrightarrow G$ is an isogeny. Then  $\rho$ is an \emph{abstract homomorphism},   in the sense of the famous \emph{Homomorphismes``abstraits''} paper by Borel and Tits \cite{Borel-Tits1973},  from $\widehat{G}$ to  $\mathop{\rm GL}_K(V)$,
\[
\rho:  \widehat{G} \longrightarrow \mathop{\rm GL}(V).
\]

By \cite[Corollary 10.4]{Borel-Tits1973} (compare with \cite{Steinberg1963}), $\rho$ is equivalent to a tensor product
\[
\varphi_1\omega_1 \otimes \cdots \otimes \varphi_m\omega_m
\]
where $\omega_i$ are rational irreducible representations of $\widehat{G}$ and $\varphi_i$ are automorphisms of  $\widehat{G}$ induced by automorphisms of the field $K$.

Now all that we have to prove is that all field automorphisms $\varphi_i$ are definable in $V \rtimes G$. For that, we have to switch from the representation-theoretic language to the group-theoretic one.

Denote by $\epsilon$ the $1$-dimensional trivial representation $\epsilon: \widehat{G} \longrightarrow \mathop{\rm GL}_1(K)$.  Set $d_i = \dim_K  \omega_i$,  then, by properties of tensor products,  $d_1 \cdots d_m = n$, where $n= \dim_K(V)$. As usual, define $d_i\epsilon = \epsilon \oplus \cdots \oplus  \epsilon$ ($d_i$ times).   Finally, set
\[
\rho_i = d_1\epsilon \otimes \cdots \otimes d_{i-1}\epsilon \otimes  \varphi_i \omega_{i} \otimes d_{i+1}\epsilon \otimes \cdots \otimes d_m\epsilon.
\]
The representations $\rho_i$ are defined up to equivalence of representations, but they can be replaced by equivalent ones in a way that makes the  images $G_i$ of $\rho_i$
pairwise commuting, $[G_i, G_j] = 1$ for $i \ne j$. Set $\breve{G} = G_1\cdots G_m$. Observe that $G\leqslant \breve{G}$. In particular, the group  $\breve{G}$ acts on $V$ irreducibly. Also, the centres $Z(G_i) < Z(\mathop{\rm GL}(V))$ consist of scalar matrices. Let $\zeta: \mathop{\rm GL}(V) \longrightarrow \mathop{\rm PGL}(V)$ is the canonical homomorphism; replacing all our groups by their images in $\mathop{\rm PGL}(V)$ simplifies the arguments. The steps of this reduction are shown on diagram (\ref{Eq:bypass})

\begin{equation}\label{Eq:bypass}
\begin{diagram}
&\widehat{G} & \rdTo^{\varphi_i\omega_i}&& \mbox{simply connected cover of } G \\
&\dTo^{\rho=\bigotimes \widetilde\varphi_i\omega_i}      &\rdTo^{\rho_i}&&     \mbox{representations}                 \\
&G   && G_i   & \mathop{\rm GL}(V)                          \\
&\dTo^{\zeta} & &\dTo^{\zeta} & \mathop{\rm GL}(V) \longrightarrow \mathop{\rm PGL}(V)  \\
&\bar{G} &\rTo^{\widetilde{\varphi}_i\bar{\omega_i}} & \bar{G}_i  & \mathop{\rm PGL}(V)  \\
 &&    \rdTo^{\pi_i}                                        & \dTo^{\rm Id}  & \mbox{projections}~\bar{G} \longrightarrow \bar{D} \\
&&                                  & \bar{G}_i  &  \mathop{\rm PGL}(V)
\end{diagram}
\end{equation}
\noindent
and explained in detail below.

We are moving now into the setup of \cite[Theorem 10.3]{Borel-Tits1973} and denote by $\bar{\rho}$ and $\bar{\rho}_i$ the induced homomorphisms
\[
\bar{\rho} = \rho\cdot\zeta: \widehat{G} \longrightarrow \mathop{\rm PGL}(V), \qquad \bar{\rho}_i = \rho_i\cdot \zeta: \widehat{G} \longrightarrow \mathop{\rm PGL}(V)
\]
We denote by $\bar{G}$ and $\bar{G}_i$ the images of groups $G$ and $G_i$ in $\bar{P}= \mathop{\rm PGL}(V)$. Since $\bar{G}_i =G_i/Z(G_i)$ are simple groups,
the commuting product $\bar{G}_1 \cdots \bar{G}_m$ is a direct product,
\[
\bar{G}_1 \cdots \bar{G}_m = \bar{G}_1 \times \cdots\times  \bar{G}_m.
\]
Now we take the double centraliser closures of groups $\bar{G}_i$, setting  $\bar{D}_i = C_{\bar{P}}(C_{\bar{P}}(\bar{G}_i))$, then $\bar{G}_i \leqslant \bar{D}_i$ and the groups $\bar{D}_i$ form a direct product
\[
\bar{D}= \bar{D}_1 \times \cdots \times \bar{D}_m.
\]
It could be shown that $\bar{D}_i \simeq  \mathop{\rm PGL}_{d_i}(K)$, but we will not be using this fact. What  matters for us is that $\bar{D}_i$  are definable in $\mathop{\rm PGL}(V)$ hence in $V\rtimes G$.

The final step in the proof starts with an observation that  $\bar{G}\leqslant \bar{D}$. The group $\bar{G}$ is of course definable, hence the projection maps
$\mathop{\pi}_i: \bar{G} \longrightarrow \bar{D}_i$ are definable in $V\rtimes G$. The image of $\pi_i$ is $\bar{G}_i$, and the triangle at the bottom of the diagram \ref{Eq:bypass} is commutative, $\pi_i = \widetilde{\varphi}_i\bar{\omega}_i$ and $\widetilde{\varphi}_i =\pi_i\bar{\omega}_i^{-1}$. Hence the field automorphism  $\widetilde{\varphi}_i$ of the group $\bar{G}$ is definable in $V\rtimes G$, hence the automorphism $\varphi$ of the field $K$ is also definable in $V\rtimes G$.
\end{proof}

\subsection{Theorem  \ref{th:linearity} part (2): an example} \label{sec:example} Let $\widehat{G}=
\mathop{\rm SL}_2(K)$, where $K$ is an algebraically closed field of characteristic
$p>2$, $\omega$ the canonical $2$-dimensional representation of $\widehat{G}$ over $K$  and $\tilde{\varphi}$ its field automorphism induced by $\varphi\in \mathop{\rm Aut} K$. Let $V$ be the space of the representation $\omega \otimes \tilde{\varphi}\omega$, then $\dim_K V = 4$. Using the usual notation $I_n$ for the identity linear transformation of $K^n$, we see that the image in $\mathop{\rm GL}(V)$ of the central element $-I_2$ of $\widehat{G}$ is
\[
-I_2\otimes -I_2^{\tilde{\varphi}} = -I_2\otimes -I_2 = I_4,
\]
which means that the image of $\widehat{G}$ in $\mathop{\rm GL}(V)$ (we denote it $G$) is isomorphic to $\mathop{\rm PSL}_2(K)$.  If we now move to $\mathop{\rm PGL}(V)$, retaining notations from the proof, we see that $\bar{G}$ now is a subgroups in $\mathop{\rm PSL}_2(K) \times \mathop{\rm PSL}_2(K)$ and that it happens to be exactly the graph of the field automorphism $\tilde{\varphi}: \mathop{\rm PSL}_2(K) \longrightarrow \mathop{\rm PSL}_2(K)$.

If we now look only at the group $V \rtimes G \simeq K^4 \rtimes \mathop{\rm PSL}_2(K)$, it would be difficult to find its representation-theoretic origins without invoking the simply connected cover $\widehat{G}$ of the group $G$.

The proof of the theorem in this special case contains an interesting little detail which sheds light at the situation in general: we are proving the definability of  the automorphism $\tilde{\varphi}$ by proving the definability of its graph  as a subgroup. This is a cute tiny self-evident  fact from elementary algebra which \c{S}\"{u}kr\"{u} Yal\c{c}\i nkaya and myself could not find in any textbook but which we systematically  use in all our work on black box algebra \cite{BY2018}:
\bq
\emph{A map $\varphi: G \longrightarrow H$ from a group $G$ to a group $H$ is a homomorphism if and only if its graph $\Gamma_\varphi \subset G\times H$ of $\varphi$  is a subgroup of $G\times H$.}
\eq
The same is of course true for rings and all other kinds of algebraic systems.  This is one of many examples of exchange of ideas between the theory of black box groups and the theory of groups of finite Morley rank -- see also Section \ref{sec:black-box-groups} for further discussion.

\subsection{Proof of Theorem \ref{th:linearity-2}} \label{sec:linearity-2}
The following result about solvable groups of finite Morley rank is an adaptation of the method of proof of Theorem \ref{th:linearity}(1).

\medskip

\textbf{Theorem \ref{th:linearity-2}.}  Let $G$ be an infinite solvable--by-finite group of finite Morley rank which acts faithfully and  definably  on a connected elementary abelian  $p$-group $V$ of finite Morley rank. Assume that this action is definably irreducible. Then $G^\circ$ is a good torus and  $V$ has a definable structure of a finite dimensional $K$-vector space compatible with the action of $G$, with the field $K$ definable in $V\rtimes G$.

\begin{proof} Since $G$ acts faithfully and irreducibly on $V$, $[G^\circ, G^\circ]$  acts trivially on $V$ by \cite[Lemma I.8.2]{ABC-book}, and since the action is faithful, $[G^\circ, G^\circ] = 1$ and $G^\circ$ is abelian. By a similar argument,  $G^\circ$ is $p^\perp$-group.  By \cite[Proposition I.11.7]{ABC-book} $G^\circ$ is a good torus.  By  \cite[Fact I.9.5]{ABC-book} there is a subgroup $G_\infty < G$ such that $G ^\circ \cap G_\infty$ is the torsion part of $G^\circ$ and $G^\circ G_\infty= G$. Obviously, $G_\infty$ is a locally finite group and $G$ is the definable closure of $G_\infty$.

Now we can repeat, with very small changes, the proof of Theorem \ref{th:linearity}(1).
\end{proof}

\section{Historical and other comments} I wish to conclude the paper with a few words about the balance of the model-theoretic and the group-theoretic components in the theory of simple groups of finite Morley rank.

\subsection{Simple algebraic groups, Chevalley groups and the work version of the  Cherlin-Zilber Conjecture}

In the classification theory of simple groups of finite Morley rank, as it stands  now, the structural theory of simple algebraic groups is heavily used, as a rule, in the form of a summary statement:
\begin{quote}
\emph{A simple algebraic group over an algebraically closed field $K$ is a Chevalley group over $K$.}
\end{quote}
A Chevalley group over a field or a ring $R$ is viewed as a group given by some specific generators and relations which involve parameters from $R$ \cite{Borel1970}. It can be shown that
\begin{quote}
\emph{A  Chevalley group over an algebraically closed field $K$ is a simple algebraic group over $K$.}
\end{quote}
In works aimed at a prooving the  Cherlin-Zilber Algebraicity Conjecture is usually used in the following form:

\begin{quote}
  \emph{An infinite  simple group of finite Morley rank is a Chevalley group over an algebraically closed field.}
\end{quote}

This allows us to use very powerful group-theoretic characterisations of Chevalley groups and ignore the algebraic geometry aspects of the theory. See, for example, how the Curtis-Tits-Phan-Lyons Theorem \cite{Gorenstein1996}, which describes Chevalley groups   as amalgams (in the group-theoretic sense) of groups of type $\mathop{\rm SL}_2$ or $\mathop{\rm PSL}_2$, is used in \cite{Berkman2008}.

Let\/ $G$ be a simple algebraic groups over an algebraically closed field $K$. Let $T$ be a maximal torus in $G$. The canonical approach to  description of $G$ as a Chevalley group  is to associate with $T$ the Weyl group, a finite system $\Phi$ of roots, root subgroups, etc. Almost all information about $G$ needed for using $G$ within an attempted proof of the Cherlin-Zilber Conjecture is contained in the set of the so-called \emph{root $\mathop{SL}_2$-subgroups} which can be characterised as Zariski closed subgroups in $G$ isomorphic to $\mathop{SL}_2(K)$ or $\mathop{PSL}_2(K)$ and normalised by $T$. They are labelled by pairs of roots $\pm r \in \Phi$, and they generate $G$. The origins of the concept go to the paper by Borel  and de Siebenthal  of 1949 on the structure of compact Lie groups \cite{Borel-deSiebenthal1949}; see the construction of root $\mathop{SL}_2$-subgroups in \cite[3.2(1)]{Borel1970}. The following special case of  the Curtis-Tits-Phan-Lyons Theorem is formulated in \cite[Proposition 2.1.]{Berkman2008}.

\begin{fact} \label{fact:amalgam}
Let\/ $\Phi$ be an irreducible finite root system of s rank at least $3$, and
let $\Pi$ be a system of fundamental roots for\/ $\Phi$. Let $X$ be a group generated by subgroups $X_r$ for
$r \in \Pi$. Suppose that either $[X_r, X_s]=1$ or  $X_{rs} = \langle X_r,X_s\rangle$  is a Chevalley group over an algebraically closed  field  with the root system  $\Phi_{rs}$ spanned by $r$ and $s$, and with $X_r$ and $X_s$ corresponding root\/ $\mathop{\rm SL}_2$-subgroups with respect to some maximal torus of $X_{rs}$.
Then $X/Z(X)$ is isomorphic to a Chevalley group with the root system $\Phi$ via a map carrying the subgroups $X_r$ to root\/
$\mathop{\rm SL}_2$-subgroups.
\end{fact}

\subsection{Central extensions of simple algebraic groups}

However there is a model-theoretic twist again: all these arguments rely on the description of central extensions of Chevalley groups in the finite Morley rank context due to Tuna Alt\i nel and Gregory Cherlin \cite{Altinel-Cherlin1999}, which, in its turn, relies on model theoretic results by Newelski and Wagner (independently):

\begin{fact} \emph{(}\cite{Newelski1991,Wagner2001},  cf. \cite[Lemma I.4.16(2)]{ABC-book}\emph{)}.
Let $K$ be a field of finite
Morley rank and $X$ a definable subgroup of $K^\times$ which contains the multiplicative group of
an infinite subfield $F$ of $K$ (not assumed definable). Then $X = K^\times$.
\end{fact}

In particular, \cite{Altinel-Cherlin1999} allows to conclude, that if the subgroup $X$ in  Fact~\ref{fact:amalgam} is of finite Morley rank, then  not only $X/Z(X)$, but $X$ itself is a Chevalley group.

\subsection{Good tori} It is worth noticing that a key ingredient of the proofs of Theorem \ref{fact:intermediate-subgroup} and of Theorem \ref{th:linearity-2}, a `\emph{good torus}', a concept introduced by Gregory Cherlin \cite{Cherlin2005}, is rooted in the model theory. In the book  \cite{ABC-book}, Proposition I.11.7, quoted in the proofs,  goes back to Proposition I.4.15, which is a deep model-theoretic  result of 2001 by Frank Wagner \cite{Wagner2001}.

\subsection{Bi-interpretability} Another key ingredient, the bi-interpretability of the simple algebraic group $G$ over an algebraically closed filed $K$, and the field $K$, is Bruno Poizat's result \cite{poizat1988} of 1988. Its proof is purely model-theoretical and does not use the structural theory or classification of simple algebraic groups. What is even more remarkable, Bruno Poizat does not even assume  that $G$ is linear -- in his paper,  the structure of $G$ as an algebraic variety is proved to be definable in the group language of $G$. Its prehistory is also interesting: the model-theoretic ideas underpinning the proof can be traced back to Zilber's result of 1977. As   Gregory Cherlin discussed it in 1979,

\begin{quote}
[\dots] \emph{Zilber\/ \cite{Zilber1977A} gives an elegant proof that
a simple algebraic group over an algebraically closed field is\/ $\aleph_1$-categorical [16, Corollary to Theorem 3.2]. I had observed (Fall 1976) that this result can be obtained easily, but at some length, from the known structure theory for such groups (using a good deal of\/ \cite{Cherlin-Reineke1976} and the generators and relations of\/ \cite{Steinberg}). Zil'ber's proof is short and uses no structure theory.} (\cite[p. 2]{Cherlin1979}, reference numbers are updated  to match the bibliography of this paper.)
\end{quote}

Not surprisingly, relations between model-theoretic properties between a Chevalley group and its field (not necessary algebraically closed) or a ring of definition are of natural interest.

At the present time, the most powerful result belongs to Elena Bunina \cite{Bunina2022}.

\begin{quote}
\emph{If\/ $G(R) = G_\pi(\Phi, R)$  is a
Chevalley group of rank\/ $> 1$, $R$ is a local ring \emph{(}with $\frac{1}{2}$ for the root systems\/ $\mathop{\rm A}_2, \mathop{\rm B}_l, \mathop{\rm C}_l, \mathop{\rm F}_4, \mathop{\rm G}_2$
and with $\frac{1}{3}$ for $\mathop{\rm G}_2$\emph{)}, then the group $G(R)$ is regularly bi-interpretable with the
ring $R$.}
\end{quote}
As usual, the Chevalley group $G_\pi(\Phi, R)$ is  constructed from the root system $\Phi$, a
ring $R$ and a representation $\pi$ of the corresponding Lie algebra \cite{Borel1970}.

But the bi-interpretability of a Chevalley group over an algebraically closed field $K$ with this field is a relatively easy result.

\subsection{Bi-interpretability in the black box algebra} \label{sec:black-box-groups} Anatoly Maltsev  was the pioneer, in 1961, of the study of bi-interpretability  of Chevalley groups and their fields of definition. Theorem 4 of his paper \cite{Maltsev1961} states the bi-interpretability of linear groups $G = \mathop{\rm GL}_{n}(K)$, $\mathop{\rm PGL}_{n}(K)$, $\mathop{\rm SL}_{n+1}(K)$, and $\mathop{\rm PSL}_{n+1}(K)$, $n \geqslant 2$ over a field $K$  and  the field $K$.

Moreover, Maltsev had shown that this bi-interpretability is recursive: there are algorithms which rewrite formulae from $\mathop{\rm Th}(G)$ as formulae from $\mathop{\rm Th}(K)$, and vice versa. This algorithmic aspect has interesting and somewhat bizarre analogues in \emph{Black Box Algebra} as developed by \c{S}\"{u}kr\"{u} Yal\c{c}\i nkaya and myself \cite{BY2024}. The black box algebra deals with finite algebraic structures (in particular, Chevalley groups over finite fields), where, however, algebraic operations are computed by  `black boxes' (these are  finite analogues  of  definable structures from model theory).  Homomorphisms have to be computed (and first order formulae evaluated) by Monte-Carlo algorithms in probabilistic polynomial time.  Polynomial time morphisms are analogues of definable homomorphisms from model theory. Interestingly, this approach gives some indication why  Maltsev skipped, in his theorem, the groups $\mathop{\rm SL}_{2}(K)$, and $\mathop{\rm PSL}_{2}(K)$: in the black box context, we do not have a direct access to non-trivial unipotent elements even in these ``small'' groups. In model theory, their existence is a basic statement
\begin{equation}
(\exists u)\left(u^p=1 \wedge u \ne 1\right), \label{eq:unipotent}
\end{equation}
but in the black box groups  the quantifier $\exists$ means ``can be found in probabilistic polynomial time'' and  proof of (\ref{eq:unipotent}) for $\mathop{\rm SL}_{2}(K)$ and $\mathop{\rm PSL}_{2}(K)$ was believed to be an intractable problem.  Even this innocently looking problem was seen as impossibly difficult:
\bq
Assume that you are given several matrices $M_1,\dots, M_m$ of size $n\times n$ over a finite field $\mathbb{F}_{p^k}$ generating a subgroup $X$ isomorphic to $\mathop{\rm SL}_{2}(\mathbb{F}_{p^\ell})$. Find in $X$ a non-trivial unipotent element (that is, an elements of order $p$).
\eq
The reason for that is simple: the probability to hit a unipotent element at random is about $\frac{1}{p^k}$ and is exponentially small with the growth of $k$, even with the small values of $p$.

\c{S}\"{u}kr\"{u} Yal\c{c}\i nkaya and myself clarified all that in  \cite{BY2018}, by constructing  $\mathop{\rm PGL}_{2}(K)$ from $\mathop{\rm PSL}_{2}(K)$ viewed as a pure group, then interpreting a projective plane $\mathbb{P}^2(K)$ in $\mathop{\rm PGL}_{2}(K)$, and,  finally,  interpreting $K$ in  $\mathbb{P}^2(K)$ (the last step is well-known but has some twists in the polynomial time setting). In bigger Chevalley groups, finding unipotent elements is done by recursion to these ``small'' cases (see Yal\c{c}\i nkaya \cite{Yalcinkaya2007,Yalcinkaya2007} and our monograph \cite{BY2024}).   This was  considerably more difficult than Maltsev's analysis in \cite{Maltsev1961} -- but still, Maltsev was the pioneer.

Another example of exchange of ideas between the black box groups theory and the  theory of groups of finite Morley rank was given in Section \ref{sec:example}. In the black box group theory, a group homomorphism $\varphi: G \longrightarrow H$ with a black box for its graph $\Gamma_\varphi < G\rtimes H$ is not necessarily polynomial time computable; we call it \emph{protomorphism}. The concept of protomorphism is central to the theory of black box groups.

\subsection{Alternative versions of proof of Theorem \ref{fact:intermediate-subgroup}} I will now outline two alternative versions of proof of Theorem \ref{fact:intermediate-subgroup} which use the structural theory of simple algebraic groups to reduce the proof  to the case of simple algebraic groups $G$ of type $A_1$, that is, $G = \mathop{PSL}_2(K)$ or  $\mathop{SL}_2(K)$.

\begin{lemma} \label{lm:enough} It suffices to prove Theorem \emph{\ref{fact:intermediate-subgroup}} for $G$ of type $G$ is of type $A_1$, that is, $G = \mathop{PSL}_2(K)$ or  $\mathop{SL}_2(K)$.
\end{lemma}

\begin{proof} Assume that we are in the setup of Theorem \ref{fact:intermediate-subgroup} and that we already know that Theorem \ref{fact:intermediate-subgroup} is true in the special case of  $G$ being of type $A_1$, that is, $G = \mathop{PSL}_2(K)$ or  $\mathop{SL}_2(K)$.

Let us pick in $G_\infty$ a maximal torus $T_\infty$ and set $T=C_G(T_\infty)$; this is maximal torus in $G$, and it equals to the Zariski closure of $T_\infty$. Denote by $\mathcal{L}_\infty$ the set of all root $\mathop{SL}_2$-subgroups in $G_\infty$ normalised by the torus $T_\infty$ and by $\mathcal{L}$ the set of their Zariski closures, then $\mathcal{L}$ is the set of all root $\mathop{SL}_2$-subgroups in $G$ normalised by the torus $T$.

Let now $L \in \mathcal{L}$, then $L_\infty = L \cap G_\infty$ is $\mathop{(P)SL}_2(K_\infty)$ and
\[
L_\infty \leqslant M \cap L \leqslant L
\]
and by the assumptions of the Lemma, $L= M\cap L < M$.  Since the system $\mathcal{L}$ generates the group $G$, $M=G$. This proves the Lemma.
\end{proof}

\begin{lemma} \label{lm: A1}
Theorem \emph{\ref{fact:intermediate-subgroup}} is true if $G$  is of type $A_1$, that is, $G = \mathop{PSL}_2(K)$ or  $\mathop{SL}_2(K)$.
\end{lemma}

\begin{proof} There are at least five approaches to a proof of this lemma.

\ben
\item This is the most direct and straightforward approach to the proof: let $T_\infty$ be a  torus in $G_\infty$ and $T = C_G(_\infty)$ is a torus in $G$. By basic linear algebra, $T$ is contained in a Borel subgroup of $G$ and is therefore a good torus by \cite[Proposition I.11.7]{ABC-book}. Hence $M \cap T = T$ and $T < M$. If $U_\infty$ and $V_\infty$ are the two maximal unipotent subgroups in $G_\infty$ normalised by $T_\infty$ then $U = [U_\infty, T]$ and  $V = [V_\infty, T]$ are  two different maximal unipotent subgroups in $G$ and belongs to $M$,  Since $\langle U, V \rangle  = G$, we have  $M = G$. \hfill $\square$

    Other four approaches are only indicated,

\item The lemma immediately  follows from Theorem 4 of Bruno Poizat's seminal paper \cite{poizat2001} (which, in its turn,  is based on the model theoretic, by their nature, results by Frank Wagner \cite{Wagner1990,Wagner2001}).

 \item \cite[Theorem 4]{poizat2001} has been drastically improved by the very neat result of the paper of Mustafin and Poizat \cite{Mustafin-Poizat2005}: a superstable non-solvable subgroup of $\mathop{\rm SL}_2(K)$ is conjugated to $\mathop{\rm SL}_2(k)$, where $k$ is an
algebraically closed subfield of $K$; it is therefore transparent that, if you assume, in addition,  that the pair of groups has a finite Morley rank, so has the pair of
fields, and $K = k$.

\item In odd characteristics, the lemma is an almost immediate consequence of the difficult and important result by Adrien Deloro and the late \'{E}ric Jaligot \cite{Deloro-Jaligot2016}. Indeed it follows from the well known properties of the group $\mathop{PSL}_2(K)$ that the subgroup $M$ satisfies the assumptions of their theorem, and therefore $M$ is isomorphic to $\mathop{PSL}_2(F)$ for some algebraically closed field $F$ (of the same characteristic as $K$, of course); after that it becomes  obvious that $M=G$.

\item In characteristic $2$, the lemma follows from \cite{Borovik-DeBonis-Nesin1994}. Moreover, it \emph{de facto} follows from an ancient result (the unbelievable 1900!) by Burnside \cite{Burnside1900}, see discussion in \cite[Sections 4 and 5]{Borovik-Mathematics-Inherited2013} and in \cite[pp.\ 11--12]{Feit1979}.
\een
To my taste, approach (1) is the simplest and best fits the needs of classification of simple groups of finite Morley rank.
\end{proof}


\section*{Acknowledgements}

My special thanks go to Ay\c{s}e Berkman, my research collaborator and co-author of many years. Problems solved in this paper originate in our joint project \cite{berkman-borovik-sharp,berkman-borovik,berkman-borovik-solvable}.

I was much influenced by ideas that my co-authors Tuna Alt\i nel, Ay\c{s}e Berkman, Jeff Burdges,  Adrien Deloro,  Gregory Cherlin, Ali Nesin, and \c{S}\"{u}kr\"{u} Yal\c{c}\i nkaya shared with me over years of our collaboration.

Gregory Cherlin kindly copy-edited an earlier version of this paper; any grammatical and orthographic errors introduced by me later are my own responsibility.  Yuri Zarhin sent to me his paper \cite{Bandman-Zarhin2023}, which extended the range of Jordan subgroups.
Theorem \ref{th:linearity} and Corollary \ref{corollary: algebraic-groups} were first proved for simple groups $G$; Gregory Cherlin suggested that the result remains valid for connected groups. Questions, comments and advice from Adrien Deloro, Gregory Cherlin, Bruno Poizat, and \c{S}\"{u}kr\"{u} Yal\c{c}\i nkaya  significantly improved the paper and led to new results included -- for example, to Theorem \ref{th:linearity} part (2).

The anonymous referee made a number of comments which helped to improve the text.

\end{document}